\documentclass[12pt]{amsart}

\newtheorem{theorem}{Theorem}
\newtheorem{lemma}[theorem]{Lemma}
\newtheorem{corollary}[theorem]{Corollary}
\newtheorem{proposition}[theorem]{Proposition}

\begin{document}

\title[A note on real algebraic groups]{A note on real algebraic groups}

\author[H. Azad]{Hassan Azad}

\address{Department of Mathematics and Statistics, King Fahd University,
Saudi Arabia}

\email{hassanaz@kfupm.edu.sa}

\author[I. Biswas]{Indranil Biswas}

\address{School of Mathematics, Tata Institute of Fundamental Research, Homi
Bhabha Road, Bombay 400005, India}

\email{indranil@math.tifr.res.in}



\date{}

\begin{abstract}
The efficacy of using complexifications to understand the structure of real
algebraic groups is demonstrated. In particular the following result is
proved. If $L$ is an algebraic subgroup of a connected real algebraic group
$G\, \subset\,\mathrm{GL}(n, {\mathbb{R}})$ such that the complexification of
$L$ contains a maximal torus of the complexification of $G$, then $L$
contains a Cartan subgroup of $G$.
\end{abstract}

\maketitle

\baselineskip=15.5pt

\section{Introduction}

The usefulness of the point of view of algebraic groups for understanding
Lie groups is demonstrated in the book of Onishchik--Vinberg \cite{OV}. In
this note we prove, in the same spirit, some results on real algebraic
groups.

If one uses this point of view, many results in real Lie groups follow
easily from known results in complex algebraic groups-for example the
existence of Cartan subgroups and the structure of real parabolic subgroups.

Our main result is the following (see Proposition \ref{prop2}):

\textit{Let $L$ be an algebraic subgroup of a connected real algebraic group
$G\, \subset\,\mathrm{GL}(n, {\mathbb{R}})$ such that the complexification
$L^{\mathbb{C}}$ of $L$ contains a maximal torus of the complexification
$G^{\mathbb{C}}$ of $G$. Then $L$ contains a Cartan subgroup of $G$.}

The following geometric result is also proved (see Proposition \ref{prop1}):

\textit{Let $G$ be a solvable real algebraic group whose eigenvalues are all
real. For any algebraic action of $G^{\mathbb{C}}$ on a complex variety, any
compact orbit of $G$ is a point.}

Recall that a subalgebra of a complex semisimple Lie algebra is a Cartan
subalgebra if it is maximal abelian and represented in the adjoint
representation by semisimple endomorphisms \cite[p. 162]{He}. This coincides
with the definition of Cartan subalgebras given in e.g \cite[p. 133]{Kn}, as
shown in \cite[p. 135]{Kn}. In this paper, a connected subgroup $C$ of a
real algebraic group $G$ is a Cartan subgroup if its complexification
$C^{\mathbb{C}}$ in the complexification $G^{\mathbb{C}}$ of $G$ is a maximal
algebraic torus in the sense of \cite[p. 133]{OV}.

\section{Real algebraic groups}

The reader is referred to \cite{OV} for general facts about real algebraic
groups. However, for technical reasons, we find it more convenient to modify
the definition of real algebraic groups as follows:

A connected subgroup $G\, \subset\, \text{GL}(n, {\mathbb{R}})$ will be
called a \textit{real algebraic group} if the connected subgroup of
$\text{GL}(n, {\mathbb{C}})$ whose Lie algebra of $\text{Lie}(G)+ \sqrt{-1}\cdot
\text{Lie}(G)$ is an algebraic subgroup of $\text{GL}(n, {\mathbb{C}})$.

The Lie algebra $\text{Lie}(G)$ will be denoted by $\mathfrak{g}$. The above
connected subgroup of $\text{GL}(n, {\mathbb{C}})$ with Lie algebra
${\mathfrak{g}}+ \sqrt{-1}\cdot{\mathfrak{g}}$ will be denoted by $G^{\mathbb{C}}$.

The group $G^{\mathbb{C}}$ is generated by the complex one-parameter
subgroups $\{\exp(zX)\,\mid\, z\,\in\, {\mathbb{C}}\}$, $X\, \in\, \mathfrak{g}$.

The connected component, containing the identity element, of the Zariski
closure of $G$ is $G^{\mathbb{C}}$, and the connected component, containing
the identity element, of the set of real points of $G^{\mathbb{C}}$ is $G$.

We will call the set of real points of a Zariski closure of a subset as the
real Zariski closure of the subset.

We have also included details that were not given in the earlier paper \cite{AB}
on solvable real algebraic groups which have all positive eigenvalues.
Classical references for this subject are the papers of Mostow \cite{Mo} and
Matsumoto \cite{Ma}.

One of the main technical tools is the following version of Lie's theorem
for real solvable groups. Recall that, by definition, a real algebraic group
is connected.

\begin{proposition}
\label{thm1} A real solvable algebraic subgroup $G\, \subset\, \mathrm{GL}
(n, {\mathbb{R}})$ is real conjugate to a subgroup of a block triangular
subgroup, whose unipotent part is the unipotent radical of a standard real
parabolic subgroup of $\mathrm{GL}(n, {\mathbb{R}})$ given by mutually
orthogonal roots and whose semisimple part is a Cartan subgroup of the same
parabolic subgroup.
\end{proposition}

\begin{proof}
By Lie's theorem, there is a common eigenvector, say $v$, with $g(v)\,=\,
\chi(g)\cdot v$ for all $g\, \in\, G$. If $\chi(g)$ is real for all $g\,
\in\, G$, then we can find a real non-zero common eigenvector $v^{\prime }$.

If $\chi(g)\,\not=\, \overline{\chi(g)}$ for some $g$, then the vectors $v$
and $\overline{v}$ are linearly independent over $\mathbb{C}$ and their
complex span $\mathcal{V}$, being conjugation invariant has a real form,
given by the real and imaginary parts of $v$. The restriction of any $g$ to
this complex span $\mathcal{V}$ is of the form
\begin{equation*}
\begin{pmatrix}
x & y \\
-y & x
\end{pmatrix}
\, .
\end{equation*}

Let $\sigma$ denote the complex conjugation of ${\mathbb{C}}^n$. For a
complex subspace $W\, \subset\, {\mathbb{C}}^n$, by $W_\sigma$ we will
denote the space of invariants for the involution $\sigma$.

Replacing the $G$--module ${\mathbb{R}}^n$ by either ${\mathbb{R}}^n/
{\mathbb{R}}\cdot v^{\prime }$ or ${\mathbb{R}}^n/({\mathbb{C}}\cdot v+
{\mathbb{C}}\cdot \overline{v})_\sigma$ depending on the above two
alternatives, and repeating the argument, we see that the matrices in $G$
can be simultaneously conjugated by some real transformation to block
triangular form, where the diagonal entries are in the multiplicative group
of positive reals or in the group of real matrices
\begin{equation*}
\{
\begin{pmatrix}
x & y \\
-y & x
\end{pmatrix}
\, \mid\, x^2+y^2 \,\not=\, 0\}\, .
\end{equation*}
If the number of $2\times 2$ blocks is $k$, then this group has $2^k$
connected components. The connected component of it containing the identity
element is a real algebraic group and in fact a Cartan subgroup of
$\text{GL}(n, {\mathbb{R}})$. Block diagonal matrices with exactly $k$ many $2\times 2$
blocks give a standard parabolic subgroup of $\text{GL}(n, {\mathbb{C}})$
corresponding to $k$ orthogonal roots in the Dynkin diagram of type $A_{n-1}$.

The real points of the unipotent radical of such a parabolic group are
normalized by the Levi complement of the corresponding real parabolic group.
In particular, they are normalized by the Cartan subgroups of the parabolic
group.
\end{proof}

\begin{lemma}
\label{lem1} If a real semisimple element of $G^{\mathbb{C}}$ has all
positive eigenvalues, then the element is in $G$.
\end{lemma}

\begin{proof}
By conjugation with a real matrix, we may assume that this real semisimple
element is a diagonal matrix $g$. Let $H$ denote the Zariski closure in
$G^{\mathbb{C}}$ of the subgroup generated by this element $g$. Then there are
finitely many characters
\begin{equation*}
\chi_j (z_1\, ,\cdots\, ,z_n)\,=\, \prod_{i=1}^n z^{m_{i,j}}_i\, ,~\
1\,\leq\, j\, \leq\, N\, , \ m_{i,j}\, \in\, {\mathbb{Z}}\, ,
\end{equation*}
of the diagonal subgroup such that
\begin{equation*}
H\, =\, \{(z_1\, ,\cdots\, ,z_n)\,\mid\, \chi_j (z_1\, ,\cdots\, ,z_n)\,=\,
1\, , \ 1\,\leq\, j\, \leq\, N\}
\end{equation*}
(see \cite[p. 17, Proposition]{Bo}).

Thus if $g$ has all positive eigenvalues, then for any $1\,\leq\, j\, \leq\,
N$,
\begin{equation*}
\chi_j(g^{1/k})\,=\, 1
\end{equation*}
for all $k\, \in\, {\mathbb{Z}}$. Therefore,
\begin{equation*}
\{\exp(t\log (g))\,\mid\, t\,\in\, {\mathbb{R}}\}
\end{equation*}
is a real one-parameter subgroup in $G^{\mathbb{C}}$ and hence it is in $G$.
\end{proof}

\begin{lemma}
\label{lem2} Any real unipotent element $u$ of $G^{\mathbb{C}}$ is in $G$.
\end{lemma}

\begin{proof}
The group $\{\exp(z\log (u))\,\mid\, z\,\in\, {\mathbb{C}}\}$ is the Zariski
closure in $G^{\mathbb{C}}$ of the group generated by $u$. Therefore, $u$
lies in the connected component, containing the identity element, of the
real points of $G^{\mathbb{C}}$. In other words, $g\, \in\, G$.
\end{proof}

\begin{corollary}
\label{cor1} If $g\, \in\, G^{\mathbb{C}}$ is real and has positive
eigenvalues, then $g$ is in $G$.
\end{corollary}

\begin{proof}
The semisimple and unipotent parts of $g$ are in $G^{\mathbb{C}}$ and are
real. Therefore, by Lemma \ref{lem1} and Lemma \ref{lem2}, they are in $G$.
Hence $g$ is in $G$.
\end{proof}

\begin{corollary}
\label{cor2} The real Zaraski closure of any subgroup of $G^{\mathbb{C}}$
generated by real elements with positive eigenvalues is connected and is
contained in $G$.
\end{corollary}

\begin{proof}
{}From \ref{cor1} it follows that the above generated group is contained in
$G$. To prove that it is connected, observe that if $g\,=\, su$ is the
decomposition of $g$ into its semisimple and unipotent parts, then the one
parameter subgroup generated by $\log(s) +\log(u)$ is in $G$ and $g$ is in
this one parameter subgroup. Now connectedness follows from Lemma
\ref{lem1}, Lemma \ref{lem2} and this observation.
\end{proof}

The following proposition will be used in the proof of Proposition
\ref{prop2}.

\begin{proposition}
\label{prop-a} If $s\, \in\, G$ is semisimple, and $H$ is the real Zariski
closure of the group generated by $s$, then $H$ has only finitely many
connected components.
\end{proposition}

\begin{proof}
Let $H_0\, \subset\, H$ be the connected component containing the identity
element. The Zariski closure of the subgroup generated by $s$ is defined by
monomial equations. Using Proposition \ref{thm1} and Lemma \ref{lem1}, we
see that if $s\, \in\, G$ is semisimple, then, modulo $H_0$, it is in the
compact part of $H$. For a compact abelian group, the real Zariski closure
coincides with the topological closure. Thus $H/H_0$ is a finite group.
\end{proof}

\begin{proposition}
\label{prop1} Let $G$ be a solvable real algebraic group whose eigenvalues
are all real. If $G^{\mathbb{C}}$ operates algebraically on a complex
variety $V$, and some $G$ orbit is compact, then this orbit is a point.
\end{proposition}

\begin{proof}
The group $G$ can be embedded into the group of real upper triangular
matrices and as $G$ is connected, it is in the connected component
containing the identity element of the group of real upper triangular
matrices. Thus all eigenvalues of the elements of $G$ are positive.

Take a point $p\, \in\, V$. The stabilizer in $G^{\mathbb{C}}$ for $p$ is an
algebraic group. Let
\begin{equation*}
G_p\, \subset\, G
\end{equation*}
be the stabilizer of $p$ for the action of $G$. Since elements of $G_p$ have
all positive eigenvalues, the subgroup $G_p$ is connected by Corollary \ref{cor2}.

Let $T$ denote the connected component containing the identity element of
the group of real upper triangular matrices. It has a filtration of real
algebraic subgroups
\begin{equation*}
T\,=\, T_0\, \supset\, T_1\, \supset\, T_2 \, \supset\,\cdots \, \supset\,
T_{\ell-1} \, \supset\, T_\ell\,=\, 1\, ,
\end{equation*}
where $T_{i+1}\, \subset\, T_i$ is a normal subgroup of codimension one. Let
$G_i\,:=\, G\bigcap T_i$ be the intersection. So we have a filtration of
algebraic subgroups
\begin{equation*}
G\,=\, G_0\, \supset\, G_1\, \supset\, G_2 \, \supset\,\cdots \, \supset\,
G_{\ell-1} \, \supset\, G_\ell\,=\, 1\, ,
\end{equation*}
such that each successive quotient is a group of dimension at most one.

The Lie algebra ${\mathfrak{g}}\,=\, \text{Lie}(G)$ is algebraic in the
sense that the semisimple and nilpotent parts of any element of
${\mathfrak{g}}$ also lie in ${\mathfrak{g}}$. Recall that $\dim G_i/G_{i+1}\,
\leq\, 1$. Hence any nonzero element $X$ in the Lie algebra of $G_i/G_{i+1}$ is
either semisimple or nilpotent. Thus as
\begin{equation*}
G_i\,=\, \{\exp(tX)\,\mid\, t\,\in\,{\mathbb{R}}\}G_{i+1}\, ,
\end{equation*}
using arguments similar to those in Lemma \ref{lem1} and Lemma \ref{lem2} we
see that
\begin{equation*}
\{\exp(tX)\,\mid\, t\,\in\,{\mathbb{R}}\}\cap G_{i+1}\,=\, 1\, .
\end{equation*}
Therefore,
\begin{equation*}
G_i\,\cong\, \{\exp(tX)\,\mid\, t\,\in\,{\mathbb{R}}\}\times G_{i+1}\, .
\end{equation*}

Thus, topologically $G$ is a cell. Here we have used that if $H$ and $K$ are
Lie subgroups of $G$ with $H\bigcap K\,=\,1$, and $HK$ is a closed subgroup
of $G$, then $HK$ is homeomorphic to $H\times K$. This implies that $H$ and
$K$ are closed in the Euclidean topology of $G$.

Let $U$ be the group of unipotents in $G_{p}$ (recall that $G_{p}\,\subset
\,G$ is the stabilizer of the point $p\,\in \,V$). Then
\begin{equation*}
G/U\,\longrightarrow \,G/G_{p}
\end{equation*}
is a fibration with total space a cell and the fiber $G_{p}/U$ is also a
cell. Now one can argue as on p 450 of Hilgert-Neeb \cite[p. 885,
Proposition 2.2]{HN}, cases 1 \& 2, to conclude that $G/G_{p}$ is a cell
and-as the orbit is compact-it must be a point. Or one can use the following
topological argument to show that if the base $G/G_{p}$ is compact, then
$G/G_{p}$ is a point. To prove this consider the long exact sequence of
homotopies for the fibration $G/U\,\longrightarrow \,G/G_{p}$. Since $G/U$
and the fiber $G_{p}/U$ are contractible, the exact sequence gives that $\pi
_{i}(G/G_{p})\,=\,0$ for all $i\,\geq \,1$. Therefore,
\begin{equation*}
H_{i}(G/G_{p},\,{\mathbb{Z}})\,=\,0
\end{equation*}
for all $i\,\geq \,1$ by the Hurewicz theorem. Since $G_{p}$ is connected,
the adjoint action of $G_{p}$ on $\text{Lie}(G)/\text{Lie}(G_{p})$ preserves
its orientation. Hence $G/G_{p}$ is orientable. We have
\begin{equation*}
H_{\dim G/G_{p}}(G/G_{p},\,{\mathbb{Z}})\,=\,\mathbb{Z}
\end{equation*}
because $G/G_{p}$ is a compact orientable manifold. Hence $\dim G/G_{p}\,=\,0
$, implying that $G/G_{p}$ is a point.
\end{proof}

A result due to Mostow \cite{Mo} and Vinberg \cite{Vi} says the following:
Any connected solvable subgroup of $G$ with all real eigenvalues is
conjugate to a subgroup of $AN$, where $G\,=\, KAN$ is an Iwasawa
decomposition of $G$. Following a suggestion of Vinberg, a proof of it is
given in \cite[p. 885, Proposition 2.2]{AB}. We note that this result can
also be proved using Proposition \ref{prop1}.

\begin{proposition}
\label{prop2} Let $L$ be an algebraic subgroup of $G$ such that the
complexification $L^{\mathbb{C}}$ contains a maximal torus of $G^{\mathbb{C}}
$. Then $L$ contains a Cartan subgroup of $G$.
\end{proposition}

\begin{proof}
Let $\mathfrak{l}$ be the Lie algebra of $L$. As $L$ is algebraic, the Lie
algebra $\mathfrak{l}$ contains both the semisimple and nilpotent part of
every element of it. If $\mathfrak{l}$ contains only nilpotent elements,
then $\mathfrak{l}$ can be conjugated over the reals to strictly upper
triangular form and the same goes for the complexification
${\mathfrak{l}}^{\mathbb{C}}$ and $L^{\mathbb{C}}$.

Thus $L$ has connected abelian subgroups consisting only of semisimple
elements. Let $S$ be such a maximal subgroup.

Working with $G/U$, where $U$ is the unipotent radical of $G$, we may assume
that $G$ is reductive.

Using Proposition \ref{thm1}, the real points of the connected component,
containing the identity element, of the Zariski closure of $S$ contains only
semisimple elements. Thus $S$ is real algebraic and its centralizer $Z(S)$
in $L$ is also real algebraic.

The group $Z(S^{\mathbb{C}})/S^{\mathbb{C}}$ is reductive algebraic and the
$Z(S)$ orbit of the point
\begin{equation*}
\xi_0\,=\, e S^{\mathbb{C}}\,\in\, Z(S^{\mathbb{C}})/S^{\mathbb{C}}
\end{equation*}
for the left--translation action is embedded as a totally real subgroup.
Also, $Z(S)/S$ fibers over the $Z(S)$ orbit of $\xi_0$ with a finite fiber.
Thus if the Lie algebra of $Z(S)/S$ is nontrivial, then it must have a
nonzero semisimple element. This would contradict the maximality of $S$.
Hence we conclude that the connected component $Z(S)^0\,=\, S$.

As the Lie algebra of $Z(S^{\mathbb{C}})$ coincides with the
complexification of the Lie algebra of $Z$, and the Lie algebra of $Z(S)$ is
also the Lie algebra of $S$, we see that $Z(S^{\mathbb{C}})\,=\, S^{\mathbb{C}}$.

Finally, $Z(S^{\mathbb{C}})$ contains a maximal torus $T$ of $L^{\mathbb{C}}$
and therefore of $G^{\mathbb{C}}$. But as $Z(S^{\mathbb{C}})\,=\, S^{\mathbb{C}}$, it
is itself an algebraic torus. Therefore, $T\,=\, S^{\mathbb{C}}$,
which means that $S$ is a Cartan subgroup of $G$.
\end{proof}

\begin{corollary}
\label{cor4} Every semisimple element of $\mathrm{Lie}(G)$ lies in some
Cartan subalgebra of $\mathrm{Lie}(G)$.
\end{corollary}

\begin{proof}
Let $X$ be a real semisimple element of $\mathrm{Lie}(G)$. Then $X$ is
tangent to the connected component $H$ of the real Zariski closure of the
one-parameter subgroup generated by $X$. Therefore, as the centralizer
$Z(H^{\mathbb{C}})$ of $H^{\mathbb{C}}$ in $G^{\mathbb{C}}$ contains a maximal
torus of $G^{\mathbb{C}}$, we conclude that the centralizer $Z(H)$ of $H$ in
$G$ contains a Cartan subgroup, say $C$, of $G$. But then $CH$ is a
connected group and the connected component, containing the identity
element, of its real Zariski closure consists of semisimple elements.
Therefore, $CH\,=\, C$ and $H$ is contained in $C$. Consequently, $X$ is in
the Lie algebra of $C$.
\end{proof}

For a different proof of Corollary \ref{cor4}, see e.g. \cite[p. 420]{He}.

Another consequence of Proposition \ref{prop2} is the following result of
\cite{Mo}.

\begin{proposition}[\protect\cite{Mo}]
\label{t1} Let $P$ be a subgroup of a real semisimple group $G$ with
$P^{\mathbb{C}}$ a parabolic subgroup of $G^{\mathbb{C}}$. Then $P$ contains a
noncompact Cartan subgroup of $G$.
\end{proposition}

\begin{proof}
By Proposition \ref{prop2}, the above subgroup $P$ contains a Cartan
subgroup, say $C$, of $G$. If $C$ is compact, then, as $P^{\mathbb{C}}$ is
conjugation invariant, and conjugation maps any root to its negative, the
unipotent radical of $P^{\mathbb{C}}$ contains a subgroup whose Lie algebra
is $\text{sl}(2, {\mathbb{C}})$. This contradiction proves that $C$ must
contain a closed real diagonalizable subgroup.
\end{proof}

As shown in Proposition \ref{t1}, the parabolic subgroup
$P$ has a noncompact Cartan subgroup. We can write it
as $AC$, where $A$ is split and $C$ compact: by Proposition \ref{thm1}, both $A$ and
$C$ are real algebraic.

Now we take a basis of the Lie algebra of $AC$, with the basis of $A$ coming
before $C$. Under conjugation, and with this lexicographic ordering of the
roots, conjugation maps every root of $Z(A)$ -- the centralizer of $A$ -- into a
negative, while positive roots supported outside the simple system of roots
of $A$ are preserved.

Since centralizers of semisimple connected groups in a complex group are
Levi-complements, one can obtain in this way the structure of real parabolic
subgroups as well as a fine classification of maximal solvable subgroups,
starting from the existence of a real diagonalizable torus -- as shown in
Matsumoto \cite{Ma} and Mostow \cite{Mo}.

Other, more recent applications of the real algebraic groups are given in
\cite{Ch}.

\section*{Acknowledgements}

The second author acknowledges the support of the J. C. Bose Fellowship.


\end{document}